\documentclass[reqno]{style}
\usepackage{boxedminipage,times,pslatex,amsmath, mathtools, graphicx, amssymb, paralist, cite, color}
\usepackage[all]{xypic}
\definecolor{e-mail}{rgb}{0,.40,.80}
\definecolor{reference}{rgb}{.20,.60,.22}
\definecolor{citation}{rgb}{0,.40,.80}
\usepackage[pdfstartview=FitH, colorlinks, linkcolor=reference, citecolor=citation, urlcolor=e-mail, backref ]{hyperref}

\newtheorem{thm}{Theorem}
\newtheorem{cor}[thm]{Corollary}
\newtheorem{lem}[thm]{Lemma}
\newtheorem{prop}[thm]{Proposition}

\theoremstyle{definition}
\newtheorem{defn}[thm]{Definition}
\theoremstyle{remark}
\newtheorem{rem}[thm]{Remark}
\numberwithin{thm}{section}
\theoremstyle{definition}

\theoremstyle{definition}

\theoremstyle{definition}
\numberwithin{equation}{section}

 \title[Computing $\mathrm{PPV}$ groups of second-order equations]{Computing the differential Galois group of a parameterized second-order linear differential equation}
 \author{Carlos E. Arreche}
\email{carreche@gc.cuny.edu}
\address{Mathematics Department, The Graduate Center of the City University of New York, New York, NY 10016}
 \thanks{This material is based upon work partially supported by a National Science Foundation (NSF) Graduate Research Fellowship (grant 40017-04-05) and by NSF grant CCF-0952591.}

\begin{document}


\begin{abstract} We develop algorithms to compute the differential Galois group $G$ associated to a parameterized second-order homogeneous linear differential equation of the form \[\tfrac{\partial^2}{\partial x^2}Y+r_1\tfrac{\partial}{\partial x}Y+r_0Y=0,\] where the coefficients $r_1,r_0\in F(x)$ are rational functions in $x$ with coefficients in a partial differential field $F$ of characteristic zero. Our work relies on the procedure developed by Dreyfus to compute $G$ under the assumption that $r_1=0$. We show how to complete this procedure to cover the cases where $r_1\neq 0$, by reinterpreting a classical change of variables procedure in Galois-theoretic terms. \end{abstract}

\maketitle


\section{Introduction} Consider a linear differential equation of the form \begin{equation}\label{equation}\delta_x^nY+\sum_{i=0}^{n-1} r_i\delta_x^iY=0,\end{equation} where $r_i\in K:=F(x)$, the field of rational functions in $x$ with coefficients in a $\Pi$-field $F$, $\delta_x$ denotes the derivative with respect to $x$, and $\Pi:=\{\partial_1,\dots,\partial_m\}$ is a set of commuting derivations. Letting $\Delta:=\{\delta_x\}\cup\Pi$, consider $K$ as a $\Delta$-field by setting $\partial_j x=0$ for each $j$. The parameterized Picard-Vessiot theory of \cite{cassidy-singer:2006} associates a parameterized Picard-Vessiot ($\mathrm{PPV}$) group to such an equation. In analogy with the Picard-Vessiot theory developed by Kolchin \cite{kolchin:1948}, the $\mathrm{PPV}$-group measures the $\Pi$-algebraic relations amongst the solutions to \eqref{equation}. The differential Galois groups that arise in this theory are linear differential algebraic groups: subgroups of $\mathrm{GL}_n$ that are defined by the vanishing of systems of polynomial differential equations in the matrix entries. The study of linear differential algebraic groups was pioneered in \cite{cassidy:1972}. The parameterized Picard-Vessiot theory of \cite{cassidy-singer:2006} is a special case of an earlier generalization of Kolchin's theory, presented in \cite{landesman:2008}, as well as the differential Galois theory for difference-differential equations with parameters \cite{hardouin-singer:2008}. 

This work addresses the explicit computation of the $\mathrm{PPV}$-group $G$ corresponding to a second-order parameterized linear differential equation \vspace{-.06in}\begin{equation}\label{original}\delta_x^2Y+r_1\delta_xY+r_0Y=0,\end{equation} where $r_1,r_0\in F(x)=:K$, and $F$ is a $\Pi$-field. In \cite{dreyfus:2011}, Dreyfus applies results from \cite{dreyfus-thesis} to develop algorithms to compute $G$, under the assumption that $r_1=0$ (see also \cite{arreche:2012} for a detailed discussion of Dreyfus' results in the setting of one parametric derivation, and \cite{arreche-unipotent} for the computation of the unipotent radical). We complete these algorithms to compute $G$ when $r_1$ is not necessarily zero. Algorithms for higher-order equations are developed in~\cite{minchenko-ovchinnikov-singer:2013a,minchenko-ovchinnikov-singer:2013b}.

After performing a change of variables on \eqref{original}, we obtain an associated equation \eqref{unimodular} of the form $\delta_x^2Y-qY=0$, whose $\mathrm{PPV}$-group $H$ is already known \cite{arreche-unipotent,dreyfus:2011}. In \S\ref{recover}, we reinterpret this change-of-variables procedure in terms of a lattice \eqref{lattice} of $\mathrm{PPV}$-fields. We recover the original $\mathrm{PPV}$-group $G$ from this lattice in Proposition~\ref{product}, which is a formal consequence of the parameterized Galois correspondence \cite[Thm.~3.5]{cassidy-singer:2006}. This reinterpretation, whose non-parameterized analogue is probably well-known to the experts in classical Picard-Vessiot theory, comprises a \emph{theoretical} procedure to recover the $\mathrm{PPV}$-group of \eqref{original} from these data.

In \S\ref{explicit}, the main tools leading to the explicit computation of $G$ obtained in Theorem~\ref{main2} are Proposition~\ref{product} and the Kolchin-Ostrowski Theorem \cite{kolchin:1968}. This strategy for computing $G$ was already sketched in \cite[\S3.4]{arreche:2012}, in the setting of one parametric derivation. The results here are sharper than those of \cite{arreche:2012}, and the proofs are conceptually simpler.

In \S\ref{examples}, we apply Theorem~\ref{main2} and the results of \cite{arreche-unipotent,dreyfus:2011} to compute the $\mathrm{PPV}$-group corresponding to a concrete parameterized second-order linear differential equation \eqref{egeq}.


\section{Preliminaries}\label{preliminaries}
We refer to \cite{kolchin:1976,vanderput-singer:2003} for more details concerning the following definitions. Every field considered in this work is assumed to be of characteristic zero. A ring $R$ equipped with a finite set $\Delta:=\{\delta_1,\dots,\delta_m\}$ of pairwise commuting derivations (i.e., $\delta_i(ab)=a\delta_i (b)+ \delta_i( a)b$ and $\delta_i\delta_j=\delta_j\delta_i$ for each $a,b\in K$ and $1\leq i,j \leq m$) is called a $\Delta$\emph{-ring}. If $R=K$ happens to be a field, we say that $(K,\Delta)$ is a $\Delta$-field. We often omit the parentheses, and simply write $\delta a$ for $\delta(a)$. For $\Pi\subseteq\Delta$, we denote the subring of $\Pi$\emph{-constants} of $R$ by $R^\Pi:=\{a\in K \ | \ \delta a=0, \ \delta\in\Pi\}$. When $\Pi=\{\delta\}$ is a singleton, we write $R^\delta$ instead of $R^\Pi$.

The \emph{ring of differential polynomials} over $K$ (in $m$ differential indeterminates) is denoted by $K\{Y_1,\dots,Y_m\}_\Delta$. As a ring, it is the free $K$-algebra in the countably infinite set of variables \begin{gather*} \{\theta Y_i \ | \ 1\leq i \leq m, \ \theta\in\Theta\}; \quad\text{where} \\ \Theta:=\{\delta_1^{r_1}\dots\delta_n^{r_n} \ | \ r_i\in\mathbb{Z}_{\geq 0} \ \ \text{for} \ \ 1\leq i \leq n\} \end{gather*} is the free commutative monoid on the set $\Delta$. The ring $K\{Y_1,\dots,Y_m\}_\Delta$ carries a natural structure of $\Delta$-ring, given by $\delta_i(\theta Y_j):=(\delta_i\cdot\theta)Y_j$. We say $\mathbf{p}\in K\{Y_1,\dots,Y_m\}_\Delta$ is a \emph{linear differential polynomial} if it belongs to the $K$-vector space spanned by the $\theta Y_j$, for $\theta\in\Theta$ and $1\leq j \leq m$. The $K$-vector space of linear differential polynomials will be denoted by $K\{Y_1,\dots,Y_m\}_\Delta^1$. The \emph{ring of linear differential operators} $K[\Delta]$ is the $K$-span of $\Theta$, and its (non-commutative) ring structure is defined by composition of additive endomorphisms of $K$.

If $M$ is a $\Delta$-field and $K$ is a subfield such that $\delta(K)\subset K$ for each $\delta\in \Delta$, we say $K$ is a $\Delta$\emph{-subfield} of $M$ and $M$ is a $\Delta$\emph{-field extension} of $K$. If $y_1,\dots,y_n\in M$, we denote the $\Delta$-subfield of $M$ generated over $K$ by all the derivatives of the $y_i$ by $K\langle y_1,\dots,y_n\rangle_\Delta$.

We say that a $\Delta$-field $K$ is $\Delta$\emph{-closed} if every system of polynomial differential equations defined over $K$ that admits a solution in some $\Delta$-field extension of $K$ already has a solution in $K$. This last notion is discussed at length in \cite{kolchin:1974}. See also \cite{cassidy-singer:2006, trushin:2010}.

We now briefly recall the main facts that we will need from the parameterized Picard-Vessiot theory \cite{cassidy-singer:2006} and the theory of linear differential algebraic groups \cite{cassidy:1972,kolchin:1985}. Let $F$ be a $\Pi$-field, where $\Pi:=\{\partial_1,\dots,\partial_m\}$, and let $K:=F(x)$ be the field of rational functions in $x$ with coefficients in $F$, equipped with the structure of $(\{\delta_x\}\cup\Pi)$-field determined by setting $\delta_xx=1$, $K^{\delta_x}=F$, and $\partial_ix=0$ for each $i$. We will sometimes refer to $\delta_x$ as the \emph{main} derivation, and to $\Pi$ as the set of \emph{parametric} derivations. From now on, we will let $\Delta:=\{\delta_x\}\cup\Pi$. Consider the following linear differential equation with respect to the main derivation, where $r_i\in K$ for each $0\leq i\leq n-1$:\begin{equation}\label{ppv-eq} \delta_x^nY+\sum_{i=0}^{n-1}r_i\delta_x^iY=0.\end{equation}

\begin{defn}  We say that a $\Delta$-field extension $M\supseteq K$ is a \emph{parameterized Picard-Vessiot extension} (or $\mathrm{PPV}$-extension) of $K$ for~\eqref{ppv-eq} if:
\begin{enumerate}[(i)]
\item There exist $n$ distinct, $F$-linearly independent elements $y_1,\dots, y_n\in M$ such that $\delta_x^ny_j+\sum_ir_i\delta_x^iy_j=0$ for each $1\leq j \leq n$.
\item $M=K\langle y_1,\dots y_n\rangle_\Delta$.
\item $M^{\delta_x}=K^{\delta_x}$.
\end{enumerate}

The \emph{parameterized Picard-Vessiot group} (or $\mathrm{PPV}$-group) is the group of $\Delta$-automorphisms of $M$ over $K$, and we denote it by $\mathrm{Gal}_\Delta(M/K)$. The $F$-linear span of all the $y_j$ is the \emph{solution space} $\mathcal{S}$. \end{defn} 

If $F$ is $\Pi$-closed, it is shown in \cite{cassidy-singer:2006} that a $\mathrm{PPV}$-extension of $K$ for \eqref{ppv-eq} exists and is unique up to $K$-$\Delta$-isomorphism. Although this assumption allows for a simpler exposition of the theory, several authors \cite{gill-gor-ov:2012, wibmer:2011} have shown that, in many cases of practical interest, the parameterized Picard-Vessiot theory can be developed without assuming that $F$ is $\Pi$-closed. In any case, we may always embed $F$ in a $\Pi$-closed field \cite{kolchin:1974,trushin:2010}. The action of $\mathrm{Gal}_\Delta(M/K)$ is determined by its restriction to $\mathcal{S}$, which defines an embedding $\mathrm{Gal}_\Delta(M/K)\hookrightarrow\mathrm{GL}_n(F)$ after choosing an $F$-basis for $\mathcal{S}$. It is shown in \cite{cassidy-singer:2006} that this embedding identifies the $\mathrm{PPV}$-group with a linear differential algebraic group (Definition~\ref{ldag-def}), and from now on we will make this identification implicitly.

\begin{defn} \label{ldag-def}Let $F$ be a $\Pi$-closed field. We say that a subgroup $G\subseteq \mathrm{GL}_n(F)$ is a \emph{linear differential algebraic group} if $G$ is defined as a subset of $\mathrm{GL}_n(F)$ by the vanishing of a system of polynomial differential equations in the matrix entries, with coefficients in $F$. We say that $G$ is $\Pi$\emph{-constant} if $G\subseteq \mathrm{GL}_n(F^\Pi)$.\end{defn}

There is a parameterized Galois correspondence \cite[Thm. 3.5]{cassidy-singer:2006} between the linear differential algebraic subgroups $\Gamma$ of $\mathrm{Gal}_\Delta(M/K)$ and the intermediate $\Delta$-fields $K\subseteq L\subseteq M$, given by $\Gamma\mapsto M^\Gamma$ and $L\mapsto \mathrm{Gal}_\Delta(M/L)$. Under this correspondence, an intermediate $\Delta$-field $L$ is a $\mathrm{PPV}$-extension of $K$ (for some linear differential equation with respect to $\delta_x$) if and only if $\mathrm{Gal}_\Delta(M/L)$ is normal in $\mathrm{Gal}_\Delta(M/K)$. The restriction homomorphism $\mathrm{Gal}_\Delta(M/K)\rightarrow\mathrm{Gal}_\Delta(L/K)$ defined by $\sigma\mapsto\sigma|_L$ is surjective, with kernel $\mathrm{Gal}_\Delta(M/L)$.

The differential algebraic subgroups of the additive and multiplicative groups of $F$, which we denote respectively by $\mathbb{G}_a(F)$ and $\mathbb{G}_m(F)$, were classified by Cassidy in \cite[Prop. 11, Prop. 31 and its Corollary]{cassidy:1972}:

\begin{prop}[(Cassidy)] \label{classification} If $B\leq \mathbb{G}_a(F)$ is a differential algebraic subgroup, then there exist finitely many linear differential polynomials $\mathbf{p}_1,\dots,\mathbf{p}_s\in F\{Y\}_\Pi^1$ such that $$B=\{b\in\mathbb{G}_a(F) \ | \ \mathbf{p}_i(b)=0 \ \text{for each} \ 1\leq i \leq s\}.$$

If $A\leq \mathbb{G}_m(F)$ is a differential algebraic subgroup, then either $A=\mu_\ell$, the group of $\ell^\text{th}$ roots of unity, or else $\mathbb{G}_m(F^\Pi)\subseteq A$, and there exist finitely many linear differential polynomials $\mathbf{q}_1,\dots,\mathbf{q}_s\in  F\{Y_1,\dots, Y_m\}_\Pi^1$ such that $$A=\left\{a\in\mathbb{G}_m(F) \ \middle| \ \mathbf{q}_{i}\bigl(\tfrac{\partial_1a}{a}, \dots, \tfrac{\partial_ma}{a}\bigr)=0 \ \text{for} \ 1 \leq i \leq s\right\}.$$
\end{prop}


\section{Recovering the original group}\label{recover}

Recall that $K:=F(x)$ is the $\Delta$-field defined by: $F=K^{\delta_x}$ is $\Pi$-closed field, $\Delta:=\{\delta_x\}\cup\Pi$, $\delta_xx=1$, and $\partial x=0$ for each $\partial\in\Pi$. Consider a second-order parameterized linear differential equation \begin{equation}\label{unimodular} \delta_x^2Y-qY=0,\end{equation} where $q\in K$. In \cite{dreyfus:2011}, Dreyfus develops the following procedure to compute the $\mathrm{PPV}$-group $H$ corresponding to \eqref{unimodular} (see also \cite{arreche:2012,arreche-unipotent}). As in Kovacic's algorithm \cite{kovacic:1986}, one first decides whether there exists $u\in\bar{K}$ such that \begin{equation}\label{factorization}(\delta_x +u)\circ (\delta_x-u)=\delta_x^2-q,\end{equation} where $\bar{K}$ is an algebraic closure of $K$. Expanding the left-hand side of \eqref{factorization} shows that such a factorization exists precisely when one can find a solution in $\bar{K}$ to the \emph{Riccati equation} $P_q(u)=\delta_xu+u^2-q=0$. One can deduce structural properties of $H$ from the algebraic degree of such a $u$ over $K$ \cite{kovacic:1986}. By \cite[Thm.~2.10]{dreyfus:2011}, precisely one of the following possibilities occurs.

\begin{enumerate}[I.]
\item \label{borel} If there exists $u\in K$ such that $P_q(u)=0$, then there exist differential algebraic subgroups $A\leq \mathbb{G}_m(F)$ and $B\leq \mathbb{G}_a(F)$ such that $H$ is conjugate to \begin{equation}\label{borel-eqn}\left\{\begin{pmatrix} a & b \\ 0 & a^{-1}\end{pmatrix} \ \middle| \ a\in A, \ b\in B\right\}.\end{equation}

\item \label{dihedral} If there exists $u\in\bar{K}$, of degree $2$ over $K$, such that $P_q(u)=0$, then there exists a differential algebraic subgroup $A\leq\mathbb{G}_m(F)$ such that $H$ is conjugate to $$\left\{\begin{pmatrix}   a & 0 \\ 0 & a^{-1}\end{pmatrix} \ \middle| \ a\in A\right\}\cup\left\{\begin{pmatrix} 0 & -a \\ a^{-1} & 0  \end{pmatrix} \ \middle| \ a\in A\right\}.$$

\item \label{finite}If there exists $u\in\bar{K}$, of degree $4$, $6$, or $12$ over $K$, such that $P_q(u)=0$, then $H$ is one of the finite primitive groups $A_4^{\mathrm{SL}_2}$, $S_4^{\mathrm{SL}_2}$, or $A_5^{\mathrm{SL}_2}$, respectively (see \cite[\S4]{singer-ulmer:1993}).

\item \label{full-sl2} If there is no $u$ in $\bar{K}$ such that $P_q(u)=0$, then there exists a subset $\Pi'\subset F\cdot\Pi$ consisting of $F$-linearly independent, pairwise commuting derivations $\partial'$ such that $H$ is conjugate to $\mathrm{SL}_2(F^{\Pi'})$.

\end{enumerate} 

The computation of $A$ in cases \ref{borel} and \ref{dihedral} is explained in \cite{dreyfus:2011}. When $A\subseteq \mathbb{G}(F^\Pi)$ is $\Pi$-constant, an effective procedure to compute $B$ in case~\ref{borel} is given in \cite[\S2, p.7]{dreyfus:2011}. This procedure is extended to the case when $A$ is not necessarily $\Pi$-constant in \cite{arreche-unipotent}. The computation of $G$ in case \ref{finite} is reduced to Picard-Vessiot theory \cite[Prop.~3.6(2)]{cassidy-singer:2006}, and the algorithms of \cite{singer-ulmer:1993} can be used to compute $G$ in this case. In case \ref{full-sl2}, for any $\partial\in\mathcal{D}:=F\cdot\Pi$ it is shown in \cite[Prop.~2.8]{dreyfus:2011} that $H$ is conjugate to a subgroup of $\mathrm{SL}_2(F^\partial)$ if and only if a certain $3^\text{rd}$-order inhomogeneous linear differential equations with respect to $\delta_x$ admits a solution in $K$, which can be decided effectively (see \cite[Prop.~2.24]{vanderput-singer:2003} and \cite[\S3]{singer:1981}). The results of \cite{gor-ov:2012} reduce the problem for general $\Pi$ to the one-parameter situation.

We now apply these results to compute the $\mathrm{PPV}$-group $G$ corresponding to \begin{equation}\label{noriginal}\delta_x^2Y-2r_1\delta_xY+r_0Y=0,\end{equation} where $r_1,r_2\in K$, and $r_1$ is not necessarily zero. The harmless normalization $-2r_1$ instead of $r_1$ will spare us the eyesore of a ubiquitous factor of $-\tfrac{1}{2}$ in what follows.

The solutions for \eqref{noriginal} are related to the solutions of an associated unimodular equation by a classical change of variables. Letting $q:=r_1^2-\delta_xr_1-r_0$, let $M$ denote a $\mathrm{PPV}$-extension of $K$ for \eqref{unimodular}, so $H=\mathrm{Gal}_\Delta(M/K)$ is the corresponding $\mathrm{PPV}$-group, which we assume has already been computed by \cite{dreyfus:2011,arreche-unipotent}. 

Let $\{\eta,\xi\}$ denote a basis for the solution space of \eqref{unimodular}, and let $U$ denote a $\mathrm{PPV}$-extension of $M$ for \begin{equation} \label{determinant} \delta_xY-r_1Y=0,\end{equation} and choose $0\neq \zeta\in U$ such that $\delta_x\zeta=r_1\zeta$. A computation shows that $\{\zeta\eta,\zeta\xi\}$ is an $F$-basis for the solution space of \eqref{noriginal}, whence $E:=K\langle\zeta\eta,\zeta\xi\rangle_\Delta\subseteq U$ is a $\mathrm{PPV}$-extension of $K$ for \eqref{noriginal}, and its $\mathrm{PPV}$-group is $G=\mathrm{Gal}_\Delta(E/K)$.

Letting $N:=K\langle\zeta\rangle_\Delta\subseteq U$, we see that $N$ is a $\mathrm{PPV}$-extension of $K$ for \eqref{determinant}, and we denote its $\mathrm{PPV}$-group by $D$ (the mnemonic is ``determinant''). The computation of $D$ is analogous to that of $A$ in case~\ref{borel} (see \cite{dreyfus:2011}). Finally, let $R:=M\cap N\subseteq U$. Since $D\leq \mathbb{G}_m(F)$ is abelian, $\mathrm{Gal}_\Delta(N/R)\leq D$ is normal, and therefore $R$ is a $\mathrm{PPV}$-extension of $K$, with $\mathrm{PPV}$-group denoted by $\Lambda$. We obtain the following lattice of $\mathrm{PPV}$-extensions and $\mathrm{PPV}$-groups.

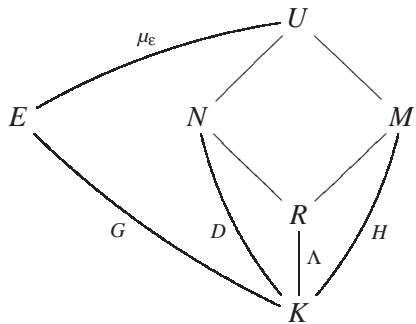
\begin{equation}\label{lattice} \xymatrix{ 
& & & U \ar@{-}[dl] \ar@{-}[dr]  \ar@/_/@{-}[dlll]_{\mu_\epsilon}\\
 E \ar@/_/@{-}[ddrrr]_(.427)G  & & N \ar@{-}[dr] \ar@/_/@{-}[ddr]_D & & M \ar@{-}[dl]  \ar@/^/@{-}[ddl]^H \\
& & & R \ar@{-}[d]^(.4)\Lambda \\ 
& & & K
} \qquad \qquad \smash{\begin{matrix}\\ \\ \\ \\ \\ \\ \\ \\ \text{where:} \\ \bullet \ E \ \text{is a} \ \mathrm{PPV}\text{-extension of} \ K \ \ \! \text{for \eqref{noriginal}}; \\ \bullet \ N \ \text{is a} \ \mathrm{PPV}\text{-extension of} \ K \ \ \! \text{for \eqref{determinant}}; \\ \bullet \  M \ \text{is a} \ \mathrm{PPV}\text{-extension of} \ K \ \text{for \eqref{unimodular}}; \\ \bullet \  U \ \text{is a} \ \mathrm{PPV}\text{-extension of} \ M \ \text{for \eqref{determinant}}; \\ \bullet \  R=M\cap N \ \ \! \text{is a} \ \mathrm{PPV}\text{-extension of} \ K.\end{matrix}}\end{equation}

\begin{rem}\label{uppv}In fact, $U=K\langle\zeta\eta,\zeta\xi,\zeta\rangle_\Delta$ is a $\mathrm{PPV}$-extension of $K$ for \begin{equation}\label{large} \delta_x^3Y-\bigl(3r_1+\tfrac{\delta_xq}{q}\bigr)\delta_x^2Y+\bigl(2r_1^2-2\delta_xr_1+r_0+2r_1\tfrac{\delta_xq}{q}\bigr)\delta_xY +\bigl(\delta_xr_0-r_1r_0-r_0\tfrac{\delta_xq}{q}\bigr)Y=0. \end{equation}To verify that each of $\zeta\eta$, $\zeta\xi$, and $\zeta$ satisfies \eqref{large}, note that $$\delta_x^2\zeta-2r_1\delta_x\zeta+r_0\zeta=-q\zeta,$$ and expand the following product in $K[\delta_x]$ to obtain the operator in \eqref{large}: $$\bigl(\delta_x-r_1-\tfrac{\delta_xq}{q}\bigr)\circ(\delta_x^2-2r_1\delta_x+r_0).$$ Let $\Gamma:=\mathrm{Gal}_\Delta(U/K)$ be the $\mathrm{PPV}$-group of $U$ over $K$. The choice of $\Delta$-field generators $\{\eta,\xi,\zeta\}$ for $U$ over $K$ produces the following embedding of $\Gamma$ in $\mathrm{GL}_3(F)$:\begin{align}\label{gamma-embedding}\gamma(\eta)&=a_\gamma\eta+c_\gamma\xi; \nonumber \\ \gamma\mapsto\smash{\begin{pmatrix} a_\gamma & b_\gamma & 0 \\ c_\gamma & d_\gamma & 0 \\ 0 & 0 & e_\gamma \end{pmatrix}}, \qquad \text{where} \qquad\, \gamma(\xi)&=b_\gamma\eta+d_\gamma\xi; \\ \gamma(\zeta)&=e_\gamma\zeta. \nonumber \end{align} Embedding $G$ in $\mathrm{GL}_2(F)$ by means of the basis $\{\zeta\eta,\zeta\xi\}$, the surjection $\Gamma\twoheadrightarrow G$ is then given by  $\gamma\mapsto e_\gamma\cdot \bigl(\begin{smallmatrix}  a_\gamma & b_\gamma \\  c_\gamma &  d_\gamma   \end{smallmatrix}\bigr)$, and therefore\begin{equation}\label{gammatog}  G\simeq \left\{\begin{pmatrix} e_\gamma a_\gamma & e_\gamma b_\gamma \\ e_\gamma c_\gamma & e_\gamma d_\gamma   \end{pmatrix} \ \middle| \ \gamma\in\Gamma\right\}.\end{equation}\end{rem}

Our next task is to apply the parameterized Galois correspondence \cite[Thm.~3.5]{cassidy-singer:2006} to the lattice \eqref{lattice} to compute $\Gamma$, and therefore $G$, in terms of $H$ and $D$. The arguments are familiar from classical Galois theory.

\begin{lem} \label{butterfly} The restriction homomorphisms \begin{align*} \mathrm{Gal}_\Delta(U/M) &\rightarrow\mathrm{Gal}_\Delta(N/R);  \qquad \text{and} \\
\mathrm{Gal}_\Delta(U/N) &\rightarrow\mathrm{Gal}_\Delta(M/R), 
\end{align*} defined respectively by $\gamma\mapsto \gamma|_N$ and $\gamma\mapsto \gamma|_M$,  are isomorphisms.
\end{lem}

\begin{proof} Since $U=M\cdot N$, these homomorphisms are injective. Since the image of $\mathrm{Gal}_\Delta(U/M)$ in $\mathrm{Gal}_\Delta(N/R)$ is Kolchin-closed, its fixed field $R'$ is an intermediate $\Delta$-field extension $R\subseteq R'\subseteq N$. Since every $f\in R'$ is fixed by every $\gamma\in\mathrm{Gal}_\Delta(U/M)$, it follows that $f\in M$, whence $f\in R$. By \cite[Thm.~3.5]{cassidy-singer:2006}, the image of $\mathrm{Gal}_\Delta(U/M)$ must be all of $\mathrm{Gal}_\Delta(N/R)$. The surjectivity of $\mathrm{Gal}_\Delta(U/N)\rightarrow\mathrm{Gal}_\Delta(M/R)$ is proved analogously.\end{proof}

\begin{prop} \label{product}The canonical homomorphism $$\Gamma\rightarrow H\times_\Lambda D:=\{(\sigma,\tau)\in H\times D \ | \ \sigma|_R=\tau|_R\},$$ given by $\gamma\mapsto (\gamma|_M,\gamma|_N)$, is an isomorphism.
\end{prop}

\begin{proof} Injectivity follows from the fact that $U=M\cdot N$. To establish surjectivity, let $\sigma\in H$ and $\tau\in D$ be such that $\sigma|_R=\tau|_R=:\lambda\in\Lambda$. Now choose $\tilde{\lambda}\in\Gamma$ such that $\tilde{\lambda}|_R=\lambda$, and define elements \begin{align*}\sigma': &=\sigma\circ\tilde{\lambda}|_M^{-1}\in\mathrm{Gal}_\Delta(M/R); \quad\text{and} \\ \tau': &=\tau\circ\tilde{\lambda}|_N^{-1}\in\mathrm{Gal}_\Delta(N/R).\end{align*} By Lemma~\ref{butterfly}, there exist $\tilde{\sigma}\in\mathrm{Gal}_\Delta(U/N)$ and $\tilde{\tau}\in\mathrm{Gal}_\Delta(U/M)$ such that $\tilde{\sigma}|_M=\sigma'$ and $\tilde{\tau}|_N=\tau'$. A computation shows that $\gamma:=\tilde{\sigma}\circ\tilde{\tau}\circ\tilde{\lambda}\in\Gamma$ satisfies $\gamma|_M=\sigma$ and $\gamma|_N=\tau$.
\end{proof}

\begin{cor}\label{lambda-trivial}The $\mathrm{PPV}$-group $\Lambda=\{1\}$ if and only if $\Gamma\simeq H\times D$. In this case, $$G\simeq\left\{\begin{pmatrix} ea & eb \\ ec & ed   \end{pmatrix} \ \middle| \ \begin{pmatrix}  a & b \\ c & d  \end{pmatrix}\in H, \ e\in D\right\}.$$\end{cor} 

We will now apply Proposition~\ref{product} to compute $G$ in cases \ref{borel}, \ref{dihedral}, \ref{finite}, and \ref{full-sl2}.


\section{Explicit computations}\label{explicit}

In case \ref{borel}, there exists a solution $u\in K$ for the Riccati equation $P_q(u)=0$. We may choose the basis $\{\eta,\xi\}$ for the solution space of \eqref{unimodular} such that $\delta_x\eta=u\eta$ and $\delta_x\bigl(\tfrac{\xi}{\eta}\bigr)=\eta^{-2}$. The embedding $H\hookrightarrow \mathrm{SL}_2(F)$ is then given by the formulae $\sigma(\eta)=a_\sigma\eta$ and $\sigma(\xi)=a_\sigma^{-1}\xi+b_\sigma\eta$ (cf. Remark~\ref{uppv}), and there are differential algebraic subgroups $A\leq \mathbb{G}_m(F)$ and $B\leq \mathbb{G}_a(F)$ such that $H$ is defined by \eqref{borel-eqn} (see \cite{arreche:2012,dreyfus:2011} for more details). We let $\Theta:=\Pi^\mathbb{N}$ denote the free commutative monoid on the set $\Pi$ (see \S\ref{preliminaries}). The $\Delta$-field $L:=K\langle\eta\rangle_\Delta$ is a $\mathrm{PPV}$-extension of $K$ for $\delta_xY-uY=0$, and $A\simeq\mathrm{Gal}_\Delta(L/K)$. We are ready to state the main result of this section.

\begin{thm} \label{main2} In case \ref{borel}, with notation as above, exactly one of the following possibilities holds: \begin{enumerate}[(i)]

\item \label{borel-1} There exist integers $k_1,k_2\in\mathbb{Z}$, with $\mathrm{gcd}(k_1,k_2)$ as small as possible, such that the image group\begin{gather*}  \{a^{k_1} \ | \ a\in A\}\neq \{1\}, \intertext{and such that there exists an element $f\in K$ satisfying}  k_1u-k_2r_1=\tfrac{\delta_xf}{f}.\end{gather*}

\item \label{borel-2} Case \eqref{borel-1} doesn't hold, and there exist linear differential polynomials $\mathbf{p},\mathbf{q}\in F\{Y_1,\dots, Y_m\}_\Pi^1$ such that the image group \begin{gather*} \bigl\{\mathbf{p}\bigl(\tfrac{\partial_1a}{a},\dots,\tfrac{\partial_ma}{a}\bigr) \ \big| \ a\in A\bigr\} \neq \{0\},\intertext{and such that there exists an element $f\in K$ satisfying} \mathbf{p}(\partial_1u,\dots,\partial_mu)-\mathbf{q}(\partial_1r_1,\dots,\partial_mr_1)=\delta_xf.\end{gather*}The $F$-vector space generated by such pairs $(\mathbf{p},\mathbf{q})$ admits a finite $F$-basis $\{(\mathbf{p}_1,\mathbf{q}_1),\dots,(\mathbf{p}_s,\mathbf{q}_s)\}$.

\item \label{borel-3}  There exist linear differential polynomials $\mathbf{p}\in F\{Y\}_\Pi^1$ and $\mathbf{q}\in F\{Y_1,\dots,Y_m\}_\Pi^1$ such that the image group \begin{gather*} \{\mathbf{p}(b) \ | \ b\in B\}\neq \{0\}, \intertext{and such that there exists an element $f\in K$ satisfying} \mathbf{p}(\eta^{-2})-\mathbf{q}\bigl(\partial_1r_1,\dots,\partial_mr_1\bigr)=\delta_xf.\end{gather*} The $F$-vector space generated by such pairs $(\mathbf{p},\mathbf{q})$ admits a finite $F$-basis $\{(\mathbf{p}_1,\mathbf{q}_1),\dots,(\mathbf{p}_s,\mathbf{q}_s)\}$.

\item \label{borel-4} $\Lambda=\{1\}$.
\end{enumerate}

Consequently, in each of these cases $G$ coincides with the subset of matrices in \begin{equation}\protect{\label{g-borel}}\left\{\begin{pmatrix} ea & eb \\ 0 & ea^{-1} \end{pmatrix} \ \middle| \ a\in A, \ b\in B, \ e\in D \right\}\end{equation} that satisfy the corresponding set of conditions below: \begin{enumerate}[(1)]
\item \label{g-borel-1}In case \eqref{borel-1}, $a^{k_1}=e^{k_2}$.
\item \protect{\label{g-borel-2}} In case \eqref{borel-2}, $\mathbf{p}_i\bigl(\tfrac{\partial_1a}{a},\dots,\tfrac{\partial_ma}{a}\bigr)=\mathbf{q}_i\bigl(\tfrac{\partial_1e}{e},\dots,\tfrac{\partial_me}{e}\bigr)$ for $1\leq i \leq s$.
\item \protect{\label{g-borel-3}} In case \eqref{borel-3}, $\mathbf{p}_i(b)=\mathbf{q}_i\bigl(\tfrac{\partial_1e}{e},\dots,\tfrac{\partial_me}{e}\bigr)$ for $1\leq i \leq s$.
\item \protect{\label{g-borel-4}} In case \eqref{borel-4}, there are no further conditions.\end{enumerate}\end{thm}

We will show slightly more in the course of the proof. We collect these facts in the form of criteria that eliminate from consideration certain possibilities from Theorem~\ref{main2} without testing them directly, based on the data of $H$ and $D$ only.

\begin{cor} \label{criteria}The following cases refer to the list of possibilities in the first part of Theorem~\ref{main2}.\begin{enumerate}[(1)] 
\item In case~\eqref{borel-1}, either $A$ and $D$ are both finite, or else they coincide as subgroups of $\mathbb{G}_m(F)$.
\item In case~\eqref{borel-2}, neither $A$ nor $D$ is $\Pi$-constant.
\item In case~\eqref{borel-3}, $A\subseteq\{\pm 1\}$, $B\neq\{0\}$, and $D$ is not $\Pi$-constant. It follows that $\eta^2\in K$, and therefore the test comprised in \eqref{borel-3} concerns elements of $K$ only.
\end{enumerate}
\end{cor}

\begin{proof}[Proof of Theorem~\ref{main2}] That the possibilities \eqref{borel-1}--\eqref{borel-4} are exhaustive and mutually exclusive will be proved in Propositions~\ref{asmall} and~\ref{alarge} below. Let us prove that each of these possibilities implies that $G$ is defined as a subset of \eqref{g-borel} by the corresponding equations contained in \eqref{g-borel-1}--\eqref{g-borel-4}, and that no more equations are required. In case~\eqref{borel-4}, the fact that $G$ coincides with \eqref{g-borel} is Corollary~\ref{lambda-trivial}.

In case~\eqref{borel-1}, a computation shows that $\eta^{k_1}\zeta^{-k_2}=cf$ for some $c\in F$. Letting $a_\gamma:=\tfrac{\gamma\eta}{\eta}$ and $e_\gamma:=\tfrac{\gamma\zeta}{\zeta}$ for $\gamma\in\Gamma$, we see that $\gamma(cf)=cf$ implies that $a_\gamma^{k_1}=e_\gamma^{k_2}$, and that $A$ is finite if and only if $D$ is finite. If $A$ and $D$ are infinite, Theorem~\ref{classification} shows that $A$ and $D$ are defined by the same linear differential polynomials (see Proposition~\ref{classification}), because \begin{equation}\label{noad}k_1\mathbf{q}\bigl(\tfrac{\partial_1a_\gamma}{a_\gamma},\dots,\tfrac{\partial_ma_\gamma}{a_\gamma}\bigr)=k_2\mathbf{q}\bigl(\tfrac{\partial_1e_\gamma}{e_\gamma},\dots,\tfrac{\partial_me_\gamma}{e_\gamma}\bigr)\end{equation} for every $\gamma\in\Gamma$ and every $\mathbf{q}\in F\{Y_1,\dots,Y_m\}_\Pi^1$, and the integers $k_1$ and $k_2$ are both different from zero. It follows from~\eqref{noad} that there are no more differential-algebraic relations defining $G$ as a subset of~\eqref{g-borel} which involve only $A$ and $D$. To see that there are no more relations at all, it suffices to show that $B$ is in the kernel of $H\twoheadrightarrow\Lambda$ (by Propostition~\ref{product}), or equivalently that $R\subseteq L$, which follows from Propositions~\ref{asmall} and \ref{alarge}. This concludes the proof of \eqref{g-borel-1}.

In case~\eqref{borel-2}, a computation shows that \begin{gather*}
\mathbf{p}\bigl(\tfrac{\partial_1\eta}{\eta},\dots,\tfrac{\partial_m\eta}{\eta}\bigr)- \mathbf{q}\bigl(\tfrac{\partial_1\zeta}{\zeta},\dots,\tfrac{\partial_m\zeta}{\zeta}\bigr)\in K.       \intertext{Since, for each $\gamma\in\Gamma$,}           \gamma\bigl(\mathbf{p}\bigl(\tfrac{\partial_1\eta}{\eta},\dots,\tfrac{\partial_m\eta}{\eta}\bigr)\bigr) = \mathbf{p}\bigl(\tfrac{\partial_1\eta}{\eta},\dots,\tfrac{\partial_m\eta}{\eta}\bigr)+\mathbf{p}\bigl(\tfrac{\partial_1a_\gamma}{a_\gamma},\dots,\tfrac{\partial_ma_\gamma}{a_\gamma}\bigr);           \intertext{and}           \gamma\bigl(\mathbf{q}\bigl(\tfrac{\partial_1\zeta}{\zeta},\dots,\tfrac{\partial_m\zeta}{\zeta}\bigr)\bigr)=\mathbf{q}\bigl(\tfrac{\partial_1\zeta}{\zeta},\dots,\tfrac{\partial_m\zeta}{\zeta}\bigr)+\mathbf{q}\bigl(\tfrac{\partial_1e_\gamma}{e_\gamma},\dots,\tfrac{\partial_me_\gamma}{e_\gamma}\bigr),          \intertext{the parameterized Galois correspondence implies that}               \mathbf{p}\bigl(\tfrac{\partial_1a_\gamma}{a_\gamma},\dots,\tfrac{\partial_ma_\gamma}{a_\gamma}\bigr)=\mathbf{q}\bigl(\tfrac{\partial_1e_\gamma}{e_\gamma},\dots,\tfrac{\partial_me_\gamma}{e_\gamma}\bigr)  \end{gather*}for each $\gamma\in\Gamma$.If $A$ (resp. $D$) were $\Pi$-constant, then $\tfrac{\partial_j\eta}{\eta}$ (resp. $\tfrac{\partial_j\zeta}{\zeta}$) would belong to $K$ for every $\partial_j\in\Pi$, which is impossible. In particular, $A$ is infinite, so Lemma~\ref{blambda} says that $B$ is in the kernel of $H\twoheadrightarrow\Lambda$. By Proposition~\ref{alarge}, in case~\eqref{borel-2} $$K\bigl\langle\tfrac{\partial_1\eta}{\eta},\dots,\tfrac{\partial_m\eta}{\eta}\bigr\rangle_\Delta \cap K \bigl\langle\tfrac{\partial_1\zeta}{\zeta},\dots,\tfrac{\partial_m\zeta}{\zeta}\bigr\rangle_\Delta=R. $$ By Propostion~\ref{product}, there are no more equations defining $G$ in \eqref{g-borel}. This concludes the proof of~\eqref{g-borel-2}.

In case~\eqref{borel-3}, a computation shows that \begin{gather*}\mathbf{p}\bigl(\tfrac{\xi}{\eta}\bigr)-\mathbf{q}\bigl(\tfrac{\partial_1\zeta}{\zeta},\dots,\tfrac{\partial_m\zeta}{\zeta}\bigr)\in K. \intertext{Let $b_\gamma:=\gamma\tfrac{\xi}{\eta}-\tfrac{\xi}{\eta}$ for each $\gamma\in\Gamma$. Since}\gamma\bigl(\mathbf{p}\bigl(\tfrac{\xi}{\eta}\bigr)\bigr)=\mathbf{p}\bigl(\tfrac{\xi}{\eta}\bigr)+\mathbf{p}(b_\gamma)\quad\qquad\text{and} \quad\qquad \gamma\bigl(\mathbf{q}\bigl(\tfrac{\partial_1\zeta}{\zeta},\dots,\tfrac{\partial_m\zeta}{\zeta}\bigr)\bigr)=\mathbf{q}\bigl(\tfrac{\partial_1\zeta}{\zeta},\dots,\tfrac{\partial_m\zeta}{\zeta}\bigr)+\mathbf{p}\bigl(\tfrac{\partial_1e_\gamma}{e_\gamma},\dots,\tfrac{\partial_me_\gamma}{e_\gamma}\bigr),\end{gather*} we see that $\mathbf{p}\bigl(\tfrac{\xi}{\eta}\bigr)\notin L$, and $\mathbf{p}(b_\gamma)= \mathbf{q}\bigl(\tfrac{\partial_1e_\gamma}{e_\gamma},\dots,\tfrac{\partial_me_\gamma}{e_\gamma}\bigr)$. Since the element $h:=\mathbf{p}\bigl(\tfrac{\xi}{\eta}\bigr)\in R$ does not belong to $L$, $B$ is not in the kernel of $H\twoheadrightarrow\Lambda$. By Lemma~\ref{blambda}, this implies that $A\subseteq\{\pm 1\}$. By Proposition~\ref{asmall}, in this case $L\cap R=K$. Therefore, the equations defining $G$ in \eqref{g-borel} do not involve $A$. This concludes the proof of \eqref{g-borel-3}.\end{proof}

In proving Theorem~\ref{main2}, it is convenient to treat separately the cases where $A\subseteq\{\pm 1\}$ and $A\nsubseteq\{\pm 1\}$. This is done in Proposition~\ref{asmall} and Propostion~\ref{alarge}, respectively. These results are obtained as consequences of the Kolchin-Ostrowski Theorem \cite{kolchin:1968}.

\begin{thm}[(Kolchin-Ostrowski)]\label{kolostro} Let $K\subseteq V$ be a $\delta_x$-field extension such that $V^{\delta_x}=K^{\delta_x}$, and suppose that $\mathbf{e}_1,\dots,\mathbf{e}_m,\mathbf{f}_1,\dots,\mathbf{f}_n\in V$ are elements such that $\tfrac{\delta_x\mathbf{e}_i}{\mathbf{e_i}}\in K$ for each $1\leq i\leq m$, and $\delta_x\mathbf{f}_j\in K$ for each $1\leq j \leq n$.

Then, these elements are algebraically dependent over $K$ if and only if at least one of the following holds:
\begin{enumerate}
\item There exist integers $k_i\in\mathbb{Z}$, not all zero, such that $\prod_{i=1}^m\mathbf{e}_i^{k_i}\in K.$

\item There exist elements $c_j\in K^{\delta_x}$, not all zero, such that $\sum_{j=1}^nc_j\mathbf{f}_j\in K.$

\end{enumerate}
\end{thm}

The following result implies Theorem~\ref{main2} in case $A\subseteq \{\pm 1\}$.

\begin{prop} \label{asmall} If $A\subseteq \{\pm 1\}$,  then $\eta^2\in K$ and exactly one of the following possibilities holds: \begin{itemize}
\item[\eqref{borel-1}] $A=\{\pm 1\}$, $D$ is finite of even order $2k$, and $$u-kr_1=\tfrac{\delta_xf}{f}$$ for some $f\in K$. Moreover, in this case $R=L$.
\item[\eqref{borel-3}] There exist linear differential polynomials $\mathbf{p}\in F\{Y\}_\Pi^1$ and $\mathbf{q}\in F\{Y_1,\dots,Y_m\}_\Pi^1$ such that the image group \begin{gather*} \{\mathbf{p}(b) \ | \ b\in B\} \neq \{ 0\}, \intertext{and such that there exists an element $f\in K$ satisfying} \mathbf{p}(\eta^{-2})-\mathbf{q}\bigl(\partial_1r_1,\dots,\partial_mr_1\bigr)=\delta_xf.\end{gather*} The $F$-vector space generated by such pairs $(\mathbf{p},\mathbf{q})$ admits a finite $F$-basis $\{(\mathbf{p}_1,\mathbf{q}_1),\dots,(\mathbf{p}_s,\mathbf{q}_s)\}$. Moreover, in this case $R\cap L=K$.
\item[\eqref{borel-4}] $\Lambda=\{1\}$.
\end{itemize}
\end{prop}

\begin{proof} If $A\subseteq \{\pm 1\}$, then $\sigma\eta=\pm\eta$ for every $\sigma\in H$, whence $\eta^2\in K$ and $L=K(\eta)$. If $\Lambda\neq\{1\}$, there exists an element $f\in R$ such that $f\notin K$, and there exist non-constant rational functions $Q_1(Y,Z_{\partial,\theta})$ and $Q_2(Y, Z_\theta)$ with coefficients in $K$, where the variables are indexed by $\partial\in\Pi$ and $\theta\in\Theta$, such that \begin{equation}\label{lambdasmall} Q_1\left(\zeta,\theta\tfrac{\partial\zeta}{\zeta}\right)=f=Q_2\left(\eta,\theta\tfrac{\xi}{\eta}\right).\end{equation} We may assume that the powers of $\zeta$ (resp. $\eta$) appearing in the numerator and denominator of $Q_1$ (resp. $Q_2$) are algebraically independent over $K$, and that the $\theta\tfrac{\partial\zeta}{\zeta}$ (resp. $\theta\tfrac{\xi}{\eta}$) appearing in $Q_1$ (resp. $Q_2$) are $F$-linearly independent. Since $\delta_x\theta\tfrac{\partial\zeta}{\zeta}\in K$ and $\delta_x\theta\tfrac{\xi}{\eta}\in K$ for each $\theta\in\Theta$ and $\partial\in\Pi$, Theorem~\ref{kolostro} implies that these $\theta\tfrac{\partial\eta}{\eta}$ (resp. $\theta\tfrac{\xi}{\eta}$) are algebraically independent over $K(\zeta)$ (resp. $L$). Clearing denominators in~\eqref{lambdasmall} shows that the elements $\zeta, \eta, \ \theta\tfrac{\partial\zeta}{\zeta}, \ \theta\tfrac{\xi}{\eta}$ are algebraically dependent over $K$. Since $\tfrac{\delta_x\zeta}{\zeta}, \ \tfrac{\delta_x\eta}{\eta}\in K$ for each $\partial\in\Pi$ and $\theta\in\Theta$, Theorem~\ref{kolostro} says that there exist integers $k_1,k_2\in\mathbb{Z}$, none of them zero, such that $\eta^{k_1}\zeta^{k_2}\in K$, or else there exist $c_\theta\in F$, almost all zero but not all zero, and $d_{j,\theta}\in F$, almost all zero but not all zero, where $1\leq j \leq m$, such that \begin{equation}\label{asmall-ops}\sum_{\theta}c_\theta\theta\tfrac{\xi}{\eta}-\sum_{j,\theta}d_{j,\theta}\theta\tfrac{\partial_j\zeta}{\zeta}\in K.\end{equation}

First suppose that $\eta^{k_1}\zeta^{k_2}\in K$ as above. Since $\Lambda\neq\{1\}$, we may assume that $k_1=1$, $\eta\notin K$, and therefore $\zeta^{k_2}\notin K$. Now $\eta\zeta^{k_2}\in K$ implies that $\zeta^{2k_2}\in K$, whence $D$ is finite. Let $k>0$ be the smallest integer $k_2$. Then $K(\zeta^k)= L=R,$ the order of $D$ is $2k$, and $$u-kr_1=\tfrac{\delta_x(\eta\zeta^{-k})}{\eta\zeta^{-k}}=\tfrac{\delta_xf}{f}.$$

Supposing instead that there are elements $c_\theta,d_{\partial,\theta}\in F$ as in \eqref{asmall-ops}, set \begin{gather*}\mathbf{p}:=\sum_\theta c_\theta\theta Y\in F\{Y\}_\Pi^1; \quad\qquad \text{and}  \qquad\quad \mathbf{q}:=\sum_{j,\theta}d_{j,\theta}\theta Y_j\in F\{Y_1,\dots,Y_m\}_\Pi^1, \intertext{so that \eqref{asmall-ops} now reads} \mathbf{p}\bigl(\tfrac{\xi}{\eta}\bigr)-\mathbf{q}\bigl(\tfrac{\partial_1\zeta}{\zeta},\dots,\tfrac{\partial_m\zeta}{\zeta}\bigr)=:f\in K. \intertext{Since $\Lambda\neq\{1\}$, we may assume that $\mathbf{p}\bigl(\tfrac{\xi}{\eta}\bigr)\notin K$, and that} \mathbf{q}\bigl(\tfrac{\partial_1\zeta}{\zeta},\dots,\tfrac{\partial_m\zeta}{\zeta}\bigr)\notin K.\intertext{This implies that $D$ is infinite, and hence connected \cite{cassidy:1972,kolchin:1985}, because whenever $D$ is finite we have that $\tfrac{\partial\zeta}{\zeta}\in K$ for each $\partial\in\Pi$. Let $b_\gamma:=\gamma\smash{\tfrac{\xi}{\eta}-\tfrac{\xi}{\eta}}$ and $e_\gamma:=\tfrac{\gamma\zeta}{\zeta}$ for each $\gamma\in\Gamma$, and note that} \gamma \bigr( \mathbf{p}\bigl(\tfrac{\xi}{\eta}\bigr)\bigr)=\mathbf{p}\bigl(\tfrac{\xi}{\eta}\bigr)+\mathbf{p}(b_\gamma). \intertext{Hence, there exists $b\in B$ such that $\mathbf{p}(b)\neq 0$. Observe that
 $\delta_x\mathbf{p}\bigl(\tfrac{\xi}{\eta}\bigr)=\mathbf{p}(\eta^{-2});$ and}
\delta_x\bigl(\mathbf{q}\bigl(\tfrac{\partial_1\zeta}{\zeta},\dots,\tfrac{\partial_m\zeta}{\zeta}\bigr)\bigr)=\mathbf{q}\bigl(\partial_1r_1,\dots,\partial_mr_1\bigr).\end{gather*}
The finite-dimensionality of the $F$-vector space generated by such pairs $(\mathbf{p},\mathbf{q})$ follows from the fact that both $N$ and $M$ have finite algebraic transcendence degree over $K$ (since $H/R_u(H)\simeq A\subseteq\{\pm 1\}$ is $\Pi$-constant\cite{minchenko-ovchinnikov-singer:2013a}), so we only need to consider finitely many elements from $\{\theta\tfrac{\xi}{\eta},\theta\tfrac{\partial\zeta}{\zeta} \ | \ \partial\in\Pi, \theta\in\Theta\}$.\end{proof}

We now consider the case when $A\nsubseteq\{\pm 1\}$. We begin with a preliminary result.

\begin{lem} \label{blambda} If $A\nsubseteq\{\pm 1\}$, then $B$ is in the kernel of the restriction homomorphism $H\twoheadrightarrow\Lambda$.
\end{lem}

\begin{proof}To show that $R\subseteq L$, we proceed by contradiction: suppose that $f\in R$ and $f\notin L$. There exist non-constant rational functions $Q_1(Y,Z_{\partial,\theta})$ and $Q_2(Z_\theta)$ with coefficients in $L$, where the variables are indexed by $\theta\in\Theta$ and $\partial\in\Pi$, such that \begin{equation} \label{lambdalarge} Q_1\bigl(\zeta,\theta\tfrac{\partial\zeta}{\zeta}\bigr)=f=Q_2\bigl(\theta\tfrac{\xi}{\eta}\bigr).\end{equation} We assume without loss of generality that the $\theta\tfrac{\partial\zeta}{\zeta}$ (resp. $\theta\tfrac{\xi}{\eta}$) appearing in $Q_1$ (resp. $Q_2$) are algebraically independent over $L$. Clearing denominators in \eqref{lambdalarge} shows that the elements $\zeta,\theta\tfrac{\partial\zeta}{\zeta},\theta\tfrac{\xi}{\eta}$ are algebraically dependent over $L$. Since $$\tfrac{\delta_x\zeta}{\zeta}, \ \delta_x\theta\tfrac{\partial\zeta}{\zeta}, \ \delta_x\theta\tfrac{\xi}{\eta}\in L$$ for each $\theta\in \Theta$ and $\partial\in\Pi$, and  since $f\notin L$, by Theorem~\ref{kolostro} there exist $c_\theta \in F=L^{\delta_x}$, almost all zero but not all zero, and $d_{\partial,\theta}\in F$, almost all zero but not all zero, such that \begin{equation}\label{blambda-1}\sum_{\theta\in\Theta} c_\theta\theta\tfrac{\xi}{\eta}+\sum_{\partial\in \Pi,\ \theta\in\Theta}d_{\partial,\theta}\theta\tfrac{\partial\zeta}{\zeta}=:g\in L.\end{equation} Applying $\delta_x$ on both sides of \eqref{blambda-1}, we obtain \begin{equation}\label{blambda-2}\sum_{\theta\in\Theta}c_\theta\theta(\eta^{-2})+\sum_{\partial\in\Pi, \ \theta\in\Theta}d_{\partial,\theta}\theta\partial r_1 = \delta_xg.\end{equation} Now choose $\sigma\in\mathrm{Gal}_\Delta(L/K)$ such that $a_\sigma:=\tfrac{\sigma\eta}{\eta}\in F^\Pi$ and $a_\sigma^2\neq 1$, and apply $(\sigma-1)$ to both sides of \eqref{blambda-2}, to obtain $$(a_\sigma^{-2}-1)\sum_{\theta\in\Theta}c_\theta\theta(\eta^{-2})=\delta_x(\sigma g - g).$$ But this implies that $\sum c_\theta\theta\tfrac{\xi}{\eta}\in L$, a contradiction. This concludes the proof that $R\subseteq L$.\end{proof}

\begin{prop} \label{alarge} If $A\nsubseteq \{\pm 1\}$, then $R\subseteq L$, and exactly one of the following possibilities holds: \begin{itemize}  
\item[\eqref{borel-1}] There exist integers $k_1,k_2\in\mathbb{Z}$, with $\mathrm{gcd}(k_1,k_2)$ as small as possible, such that the image group\begin{gather*}  \{a^{k_1} \ | \ a\in A\}\neq \{0\},\intertext{and such that there exists an element $f\in K$ satisfying}  k_1u-k_2r_1=\tfrac{\delta_xf}{f}.\end{gather*} 
\item[\eqref{borel-2}] Case \eqref{borel-1} doesn't hold, and there exist linear differential polynomials $\mathbf{p},\mathbf{q}\in F\{Y_1,\dots, Y_m\}_\Pi^1$ such that the image group \begin{gather*} \bigl\{\mathbf{p}\bigl(\tfrac{\partial_1a}{a},\dots,\tfrac{\partial_ma}{a}\bigr) \ \big| \ a\in A\bigr\}\neq\{0\}, \intertext{and such that there exists an element $f\in K$ satisfying} \mathbf{p}(\partial_1u,\dots,\partial_mu)-\mathbf{q}(\partial_1r_1,\dots,\partial_mr_1)=\delta_xf.\end{gather*}The $F$-vector space generated by such pairs $(\mathbf{p},\mathbf{q})$ admits a finite $F$-basis $\{(\mathbf{p}_1,\mathbf{q}_1),\dots,(\mathbf{p}_s,\mathbf{q}_s)\}$. Moreover, in this case \begin{equation}\label{radditive} K\bigl\langle\tfrac{\partial_1\eta}{\eta},\dots,\tfrac{\partial_m\eta}{\eta}\bigr\rangle_\Delta \cap K \bigl\langle\tfrac{\partial_1\zeta}{\zeta},\dots,\tfrac{\partial_m\zeta}{\zeta}\bigr\rangle_\Delta=R.\end{equation} 
\item[\eqref{borel-4}] $\Lambda=\{1\}$. \end{itemize}
\end{prop}

\begin{proof} Since $A\nsubseteq \{\pm 1\}$, Lemma~\ref{blambda} says that $R\subseteq L$. Assume that $\Lambda\neq\{1\}$, and let $f\in R$ such that $f\notin K$. Then there exist non-constant rational functions $Q_1(Y,Z_{\partial,\theta})$ and $Q_2(Y,Z_{\partial,\theta})$ with coefficients in $K$, where the variables are indexed by $\partial\in\Pi$ and $\theta\in\Theta$, such that \begin{equation}\label{alargeeq}Q_1(\eta,\theta\tfrac{\partial\eta}{\eta})=f=Q_2(\zeta,\theta\tfrac{\partial\zeta}{\zeta}).\end{equation} We assume without loss of generality that the $\theta\tfrac{\partial\eta}{\eta}$ (resp. $\theta\tfrac{\partial\zeta}{\zeta}$) appearing in $Q_1$ (resp. $Q_2$) are algebraically independent over $K$ (cf. the proof of Proposition~\ref{asmall}). Clearing denominators in \eqref{alargeeq} shows that the elements $\eta,\ \zeta,\ \theta\tfrac{\partial\eta}{\eta},\ \theta\tfrac{\partial\zeta}{\zeta}$ are algebraically dependent over $K$. Since \begin{gather*}\tfrac{\delta_x\eta}{\eta},\ \tfrac{\delta_x\zeta}{\zeta},\ \delta_x\theta\tfrac{\partial\eta}{\eta},\ \delta_x\theta\tfrac{\partial\zeta}{\zeta} \in K; \end{gather*} Theorem~\ref{kolostro} implies that either there are integers $k_1,k_2\in\mathbb{Z}$, none of them zero, such that $\eta^{k_1}\zeta^{-k_2}\in K$, or else there exist $c_{j,\theta}\in F$, almost all zero but not all zero, and $d_{j,\theta}\in F$, almost all zero but not all zero, where $1\leq j \leq m$, such that\begin{equation}\label{alarge-ops} \sum_{j,\theta}c_{j,\theta}\theta\tfrac{\partial_j\eta}{\eta}-\sum_{j,\theta}d_{j,\theta}\theta\tfrac{\partial_j\zeta}{\zeta}\in K.\end{equation}

If $\eta^{k_1}\zeta^{-k_2}=:f\in K$ for $k_1,k_2\in\mathbb{Z}$ as above, since $\Lambda\neq\{1\}$ we may assume that $\eta^{k_1}\notin K$. Hence, there exists $a\in A$ such that $a^{k_1}\neq 1$, and \begin{equation*}k_1u-k_2r_1=\tfrac{\delta_x(\eta^{k_1}\zeta^{-k_2})}{\eta^{k_1}\zeta^{-k_2}}=\tfrac{\delta_xf}{f}.\end{equation*} If such integers $k_1$ and $k_2$ do not exist, then \eqref{radditive} is satisfied.

Now suppose there are elements $c_{j,\theta},d_{j,\theta}\in F$ as in \eqref{alarge-ops}, and set \begin{gather*}\mathbf{p} :=\sum_{j,\theta}c_{j,\theta}\theta Y_j \in F\{Y_1,\dots,Y_m\}_\Pi^1;\qquad\quad\text{and} \quad\qquad \mathbf{q} := \sum_{j,\theta}d_{j,\theta}\theta Y_j\in F\{Y_1,\dots,Y_m\}_\Pi^1, \intertext{so that \eqref{alarge-ops} now reads} \mathbf{p}\bigl(\tfrac{\partial_1\eta}{\eta},\dots,\tfrac{\partial_m\eta}{\eta}\bigr)-\mathbf{q}\bigl(\tfrac{\partial_1\zeta}{\zeta},\dots,\tfrac{\partial_m\zeta}{\zeta}\bigr)\in K.\intertext{Since $\Lambda\neq\{ 1\}$, we may assume that} \mathbf{p}\bigl(\tfrac{\partial_1\eta}{\eta},\dots,\tfrac{\partial_m\eta}{\eta}\bigr) \notin K \quad\qquad\text{and} \quad\qquad\mathbf{q}\bigl(\tfrac{\partial_1\zeta}{\zeta},\dots,\tfrac{\partial_m\zeta}{\zeta}\bigr) \notin K. \intertext{Since, for each $\gamma\in \Gamma$,}
\gamma \bigl( \mathbf{p}\bigl(\tfrac{\partial_1\eta}{\eta},\dots,\tfrac{\partial_m\eta}{\eta}\bigr)\bigr)= \mathbf{p}\bigl(   \tfrac{\partial_1\eta}{\eta},\dots,\tfrac{\partial_m\eta}{\eta} \bigr)+\mathbf{p}\bigl(\tfrac{\partial_1a_\gamma}{a_\gamma},\dots,\tfrac{\partial_ma_\gamma}{a_\gamma}\bigr), \intertext{there exists $a\in A$ such that $\mathbf{p}\bigl(\tfrac{\partial_1a}{a},\dots,\tfrac{\partial_ma}{a}\bigr)\neq 0$. Note that} \delta_x\bigl(\mathbf{p}\bigl(\tfrac{\partial_1\eta}{\eta},\dots,\tfrac{\partial_m\eta}{\eta}\bigr)\bigr) =\mathbf{p}\bigl(\partial_1u,\dots,\partial_mu\bigr)\qquad\quad \text{and} \quad\qquad \delta_x\bigl(\mathbf{q}\bigl(\tfrac{\partial_1\zeta}{\zeta},\dots,\tfrac{\partial_m\zeta}{\zeta}\bigr)\bigr) =\mathbf{q}\bigl(\partial_1r_1,\dots,\partial_mr_1\bigr).\end{gather*}

The finite-dimensionality of the $F$-vector space generated by such pairs $(\mathbf{p},\mathbf{q})$ follows from the fact that both $L$ and $N$ have finite algebraic transcendence degree over $K$, so we only need to consider finitely many elements from the set $\{\theta\tfrac{\partial\eta}{\eta},\theta\tfrac{\partial\zeta}{\zeta} \ | \ \partial\in \Pi, \ \theta\in\Theta\}$. This concludes the proof of Proposition~\ref{alarge}.\end{proof}

We will now apply Proposition~\ref{product} to compute $G$ in case \ref{dihedral}. Recall there exists a solution $u$ to the Riccati equation $P_q(u)=0$ such that $u$ is quadratic over $K$. We denote by $\bar{u}$ the unique Galois conjugate of $u$, and set $w:=u-\bar{u}$. Then $w^2\in K$, so $\tfrac{\delta_xw}{w}=:v\in K$. There is a differential algebraic subgroup $A\leq\mathbb{G}_m(F)$ such that $H\simeq A\rtimes\{\pm 1\}$.

\begin{prop} \label{dihedral-comp} In case \ref{dihedral}, with notation as above, exactly one of the following possibilities holds: \begin{enumerate}[(i)]
\item \label{dihedral-1} $D$ is finite of even order $2k$, and $v-kr_1=\tfrac{\delta_xf}{f}$ for some $f\in K$. \item \label{dihedral-2} $\Lambda=\{1\}$. \end{enumerate}

Consequently, in each of these cases $G$ coincides with the subset of matrices in $$\left\{ \begin{pmatrix} e_1a & 0 \\ 0 & e_1a^{-1} \end{pmatrix}, \begin{pmatrix} 0 & -e_2a \\ e_2a^{-1} & 0 \end{pmatrix} \ \middle| \ a\in A; \ e_1,e_2\in D\right\}$$ that satisfy the corresponding set of conditions below:\begin{enumerate}[(1)] \item \protect{\label{g-dihedral}}In case \eqref{dihedral-1}, $e_1^k=1$ and $e_2^k=-1$. \item In case \eqref{dihedral-2}, there are no further conditions.
\end{enumerate}\end{prop}

\begin{proof} Since the commutator subgroup $[H,H]$ of $H$ coincides with $\bigl\{\bigl(\begin{smallmatrix} a & 0 \\ 0 & a^{-1}\end{smallmatrix}\bigr) \ | \ a\in A\bigr\}$, and $\Lambda$ is abelian, the surjection $H\twoheadrightarrow \Lambda$ factors through $H/[H,H]\simeq\{\pm 1\}$, the $\mathrm{PPV}$-group of the quadratic subextension $K(u)\subset M$. Therefore, $R\subseteq K(u)$ and $\Lambda$ is finite of order at most $2$.

If $D\leq\mathbb{G}_m(F)$ is infinite, then it is also connected \cite{cassidy:1972}, so its only finite quotient is $\Lambda=\{1\}$. If $D$ is finite of odd order, then $\Lambda=\{1\}$ is the only common quotient of $\{\pm 1\}$ and $D$. Hence, if $\Lambda\neq\{1\}$, $D$ must be finite of even order $2k$. Since $D=\mu_{2k}$ is cyclic, the field $K(\zeta^k)$ is the only quadratic subextension of $K(\zeta)$. By Theorem~\ref{kolostro} or classical Galois theory, $K(w)=K(\zeta^k)$ if and only if $w\zeta^{-k}\in K$. Otherwise, $R=K(w)\cap K(\zeta^k)$ coincides with $K$, which is impossible. If $w\zeta^{-k}:=f\in K,$ we see that $$v-kr_1=\tfrac{\delta_x(w\zeta^{-k})}{w\zeta^{-k}}=\tfrac{\delta_xf}{f}.$$ Letting $a_\gamma:=\tfrac{\gamma w}{w}$ and $e_\gamma:=\tfrac{\gamma\zeta}{\zeta}$ for $\gamma\in\Gamma$, we see that $a_\gamma e_\gamma^{-k}=1$. In other words, $\gamma w=w$ if and only if $e_\gamma^k=1$, and $\gamma w=-w$ if and only if $e_\gamma^k=-1$. To conclude the proof of \eqref{g-dihedral}, note that the elements of $H$ that fix $w$ are precisely those of the form $\bigl(\begin{smallmatrix} a & 0 \\ 0 & a^{-1}\end{smallmatrix}\bigr)$ for $a\in A$ \cite{kovacic:1986,ulmer-weil:2000}.
\end{proof}

\begin{rem} \protect{\label{finite-comp}} In case \ref{finite}, $H$ is a finite subgroup of $\mathrm{SL}_2(F)$. If $D\leq\mathbb{G}_m(F)$ is infinite, then $\Lambda=\{1\}$, since it is finite and connected. If $D=\mu_s$ is finite, then $U$ is algebraic over $K$, and therefore so is $E$. By \cite[Prop. 3.6(2)]{cassidy-singer:2006}, $G$ coincides with the (non-parameterized) $\mathrm{PV}$-group of \eqref{noriginal}. We only sketch the computation of $G$ in this case.

For each factor $\ell$ of $s$ and each character $\chi:H\rightarrow \mu_{\ell}$ of order $\ell$, there is an element $w_\chi\in M$ such that $K(w_\chi)\subset M$ is cyclic of order $\ell$ and $\chi(\sigma)=\tfrac{\sigma w_\chi}{w_\chi}$ for each $\sigma\in H$. Thus $K(w_\chi)$ is the fixed field of $\mathrm{ker}(\chi)$. Such an element $w_\chi$ can be computed effectively (cf. the \emph{semi-invariants} discussed in \cite[\S4.3.1]{singer-ulmer:1993}). Let $$v_\chi:=\tfrac{\delta_xw_\chi}{w_\chi}\in K.$$ If there exist integers $0<k_1<\ell$ and $0<k_2<\tfrac{s}{\ell}$ such that $$k_1v_\chi-k_2r_1=\tfrac{\delta_xf}{f}$$ for some $f\in K$, then $$G\simeq\{(\sigma,e)\in H\times D \ | \ \chi(\sigma)^{k_1}=e^{k_2}\}.$$ If no such $k_1$ and $k_2$ can be found for any $\chi\in H^*$, the character group of $H$, then $\Lambda=\{1\}$.

When $H$ is finite in cases \ref{borel} and \ref{dihedral} (i.e., $A$ is finite and $B=0$), the computation of $G$ performed in Theorem~\ref{main2} and Proposition~\ref{dihedral-comp} coincides with the one just described. \end{rem}

\begin{rem}\protect{\label{sl2-comp}} In case \ref{full-sl2}, there is a finite subset $\Pi'$ of the $F$-span of $\Pi$ consisting of $F$-linearly independent, pairwise commuting derivations such that $H$ is isomorphic to the simple group $\mathrm{SL}_2(F^{\Pi'})$ \cite{cassidy:1972}. Therefore, the only abelian quotient of $H$ in this case is $\Lambda=\{1\}$.\end{rem}


\section{Example}\label{examples}

We let $K=F(x)$ denote the $\Delta$-field of the previous sections, where $\Pi:=\{\partial_1,\partial_2\}$, $\partial_j:=\tfrac{\partial}{\partial t_j}$ for $j=1,2$, and $F$ denotes a $\Pi$-closed field containing $\mathbb{Q}(t_1,t_2)$ \cite{kolchin:1974,trushin:2010}. We now compute the $\mathrm{PPV}$-group $G$ corresponding to \begin{equation}\label{egeq} \delta_x^2Y-2\Bigl( \tfrac{t_1-t_2}{x}+\tfrac{t_2}{x-1}\Bigr)\delta_xY+ \Bigl(\tfrac{(t_1-2t_2)(t_2-1)+2(t_1-t_2)^2x}{x^2}    +    \tfrac{t_1(2t_2-t_1+1)-2(t_1-t_2)^2(x-1)}{(x-1)^2}\Bigr)Y=0.\end{equation}Note that $r_1=\tfrac{t_1-t_2}{x}+\tfrac{t_2}{x-1}$, and the coefficient $q$ in the unimodular equation \eqref{unimodular} associated to \eqref{egeq} in this case is \vspace{-.06in}$$q=\tfrac{t_1(t_1-1)(1-2x)}{x^2} + \tfrac{(t_1-t_2)(2t_1x-t_1-t_2-1)}{(x-1)^2}.$$ The Riccati equation $\delta_xu+u^2=q$ admits the solution $u=\tfrac{t_1}{x}+\tfrac{t_1-t_2}{x-1}. $ Hence, we are in case~\ref{borel}, and there are differential algebraic subgroups $A\leq\mathbb{G}_m(F)$ and $B\leq\mathbb{G}_a(F)$ such that the $\mathrm{PPV}$-group $H$ for \eqref{unimodular} is defined by \eqref{borel-eqn}. By \cite{dreyfus:2011}, for every $a\in A$ and every linear differential polynomial $\mathbf{q}\in F\{Y_1,Y_2\}_\Pi^1$, \begin{equation}\label{computa} \mathbf{q}\bigl(\tfrac{\partial_1a}{a},\tfrac{\partial_2a}{a})=0 \qquad \Longleftrightarrow \qquad \mathbf{q}(\partial_1u,\partial_2u)\in\delta_x(K).\end{equation}Since $\partial_1u=\tfrac{1}{x}+\tfrac{1}{x-1}$ and $\partial_2u=-\tfrac{1}{x-1}$, we see that $$A=\bigl\{a\in\mathbb{G}_m(F) \ \big| \ \partial_j\bigl(\tfrac{\partial_1a}{a}\bigr)=0=\partial_j\bigl(\tfrac{\partial_2a}{a}\bigr) \ \text{for} \ j=1,2\bigr\}.$$ A similar computation shows that the $\mathrm{PPV}$-group $D$ for \eqref{determinant} is defined by the same linear differential polynomials as $A$.

To compute $B$, let $H'$ denote the (non-parameterized) $\mathrm{PV}$-group of \eqref{unimodular}. The unipotent radical $B'$ of $H'$ is $\mathbb{G}_a(F)$, by Kovacic's algorithm \cite{kovacic:1986} (see also \cite[proof of Cor.~3.3]{arreche:2013} for a similar computation). Since the only derivation $$\partial\in F\cdot\partial_1\oplus F\cdot\partial_2$$ such that $\partial u\in\delta_x(K)$ is $\partial=0$, the main result of \cite{arreche-unipotent} implies that $B=B'=\mathbb{G}_a(F)$.

Having computed $H$ and $D$, we apply Theorem~\ref{main2} to compute $G$. For any integers $k_1,k_2\in\mathbb{Z}$, \begin{equation}\label{residue} k_1u-k_2r_1 =\tfrac{t_1(k_1-k_2)+t_2k_2}{x}+\tfrac{t_1k_1-t_2(k_1-k_2)}{x-1}.\end{equation} If $k_1u-k_2r_1=\tfrac{\delta_xf}{f}$ for some $f\in K$, the residues of~\eqref{residue} are integers, which is impossible unless $k_1=0=k_2$, and therefore case~\eqref{borel-1} of Theorem~\ref{main2} doesn't hold. Now the relations \begin{equation*}\partial_1u+\partial_2u = \partial_1r_1 \qquad\quad \text{and} \quad\qquad \partial_1r_1+\partial_2r_1 = -\partial_2u,\end{equation*} correspond to the linear differential polynomials \begin{align*} \mathbf{p}_1 &:=Y_1+Y_2;      &        \mathbf{p}_2 &:=-Y_2;  \\ \mathbf{q}_1 &:=Y_1;     &        \mathbf{q}_2 &:=Y_1+Y_2.\end{align*} Letting $a\in A$ such that $\partial_1a=0$ and $\partial_2a=a$, we see that $\mathbf{p}_i\bigl(\tfrac{\partial_1a}{a},\tfrac{\partial_2a}{a}\bigr)\neq 0$ for $i=1, 2$, and we have verified the conditions of Theorem~\ref{main2}\eqref{borel-2}.

Since $\theta\partial_ju=0=\theta\partial_jr_1$ for each $\theta\in\Theta$ and $1\leq j\leq2$, the set\begin{gather*}\bigl\{(\mathbf{p}_1,\mathbf{q}_1); \ (\mathbf{p}_2,\mathbf{q}_2)\bigr\} \intertext{forms a basis for the $F$-vector space of pairs $(\mathbf{p},\mathbf{q})$, with $\mathbf{p},\mathbf{q}\in F\{Y_1,Y_2\}_\Pi^1$, such that} \mathbf{p}(\partial_1u,\partial_2u)-\mathbf{q}(\partial_1r_1,\partial_2r_1)\in\delta_x(K). \intertext{Therefore, the $\mathrm{PPV}$-group for \eqref{egeq} is} G\simeq \left\{ \begin{pmatrix}   ea & eb \\ 0 & ea^{-1} \end{pmatrix} \ \middle| \ \begin{matrix} a,e \in A; \ b \in B; \\ \\ \tfrac{\partial_1a}{a}+\tfrac{\partial_2a}{a}=\tfrac{\partial_1e}{e}; \\ \\ \tfrac{\partial_1e}{e}+\tfrac{\partial_2e}{e}=-\tfrac{\partial_2a}{a}\end{matrix}\right\}.\end{gather*}


\bibliographystyle{spmpsci} 

\end{document}